\documentclass{article}
\usepackage[utf8]{inputenc}
\usepackage[english]{babel}
\usepackage{hyperref}
\usepackage{titling}
\usepackage{amsmath}
\usepackage{amsthm,amssymb}
\newtheorem{theorem}{Theorem}
\newtheorem{lemma}{Lemma}
\newtheorem{corollary}{Corollary}
\newtheorem{proposition}{Proposition}
\newtheorem{definition}{Definition}
\theoremstyle{definition}
\newtheorem{exmp}{Example}[section]

\usepackage{mathrsfs}
\usepackage{pgfplots}
\usepackage{bbm}
\usepackage{amsfonts}
\usepackage{latexsym}
\usepackage{enumerate}
\usepackage{wasysym}
\usepackage{lipsum}
\usepackage [english]{babel}
\usepackage [autostyle, english = american]{csquotes}
\MakeOuterQuote{"}
\newcommand{\interior}[1]{%
  {\kern0pt#1}^{\mathrm{o}}%
}

\title{State-Driven Dynamic Graphon Model}
\author{Shizhou Xu, Quanyan Zhu}
\date{August 2020}

\begin{document}

\maketitle

\begin{abstract}
This paper shows the equivalence class definition of graphons hinders a direct development of dynamics on the graphon space, and hence proposes a state-driven approach to obtain dynamic graphons. The state-driven dynamic graphon model constructs a time-index sequence of the permutation-invariant probability measures on the universal graph space by assigning i.i.d. state random processes to $\mathbbm{N}$ and edge random variables to each of the unordered integer pairs. The model is justified from three perspectives: graph limit definition preservation, genericity, and analysis availability. It preserves the graph limit definition of graphon by applying a bijection between the permutation-invariant probability measures on the universal graph space and the graphon space to obtain the dynamic graphon, where the existence of the bijection is proved. Also, a generalized version of the model is proved to cover the graphon space by an application of the celebrated Aldous-Hoover representation, where generalization is achieved by adding randomness to the edge-generating functions. Finally, analysis of the behavior of the dynamic graphon is shown to be available by making assumptions on the state random processes and the edge random variables.
\end{abstract}

\newpage

\section{Introduction}
Randomly networked systems, especially the large ones with unknown number of vertices, have been an active research area in different domains. Since the Erdős–Rényi random graph model \cite{ER} introduced in 1959, random graph models became an important mathematical technique in the study of arbitrarily large graphs. For example, the Erdős–Rényi model has been used to the study of phase transition problems in economics and physics, and to the epidemic spread modeling in biology $\cite{YRYP}$. The stochastic block models were developed to study the latent in social networks $\cite{HLL}$ in 1983. The main reason behind such wide range of applications of random graph models is that the resulting distributions on graphs make no apriori assumption on the number of vertices.\\

Since 2006, a sequence of papers by Lovasz and his coauthors $\cite{LL,BCL}$ developed an elegant graph limit theory, which systematically studied the distributions without apriori assumption on the number of vertices. The theory defines the convergence of a graph sequence and names the limit object graphons, which is denoted by $(\mathcal{W},\delta_{\square})$, where $\mathcal{W} := \{w: [0,1]^2 \rightarrow [0,1]| \text{ be borel-measurable}\}$ and the cut distance $\delta_{\square}$ is defined as the following: for any $w_1, w_2 \in \mathcal{W}$, $$\delta_{\square}(w_1,w_2):= \inf_{\phi\in \mathcal{S}([0,1])} ||w_1 - w_2^{\phi}||_{\square},$$ where $\mathcal{S}([0,1])$ denotes all the measure preserving maps from $[0,1]$ to $[0,1]$, $w_2^{\phi}(x,y) := w_2(\phi(x),\phi(y))$, and $$||w_1 - w_2||_{\square} := \sup_{S,T \subset [0,1]} |\int_{S \times T} w_1(x,y) - w_2(x,y) dxdy|.$$ Intuitively, each graphon is a random graph with countably infinite number of vertices. For any simple graph $F = (V,E)$, and any graphon $w$, the homomorphism density function $$t(F,w):= \int_{[0,1]^{|V|}} \prod_{(i,j) \in E}w(x_i,x_j) \prod_{i=1}^{|V|} dx_i$$ gives the probability of drawing $|V|$ number of nodes from $w$ and obtaining a graph that is isomorphic to $F$. We will give detailed explanation for $t(F,w)$ in the next section. Mathematically, graphon is the weak limit of a graph sequence with respect to the homomorphism density functions. The cut distance $\delta_{\square}$ defines graphons as equivalent classes on $\mathcal{W}$.\\

Equivalent constructions to graphon have also been developed in probability theory, actually long before the graphon. Examples includes the permutation-invariant measure on the graph universal space in $\cite{TT}$ and the exchangeable random arrays in $\cite{TA}$. The importance of graphon and its equivalents is bifold: (1) it is a distribution on graphs without apriori assumption on the number of vertices; (2) It is invariant under permutation and does not differentiate isomorphic graphs. Therefore, graphon is a rigorous tool to study arbitrarily large randomly networked system, and the permutation-invariance guarantees its value in the study of unlabeled and randomly sampled graphs.\\

In reality, interactive particles in physics, connected vertices on social networks, and data manifolds on ambient space all tend to vary over time. Moreover, due to the large number of nodes on the time-variant physics models, social networks, and data manifolds, a generic dynamic graphon model is of particular interest for those who study large networks in a dynamic setting. The difficulty of dynamic graphon model design is trifold:\\

\begin{itemize}
    \item (1. definition preservation) the model has to generalize the finite random graph models and then converges to graphons as the number of nodes goes to infinity;
    \item (2. genericity) the model needs to cover all graphon space at each time stage;
    \item (3. analysis availability) the model has to provide rich structure and analytic results for resulting the dynamical systems on $(\mathcal{W},\delta_{\square})$.\\
\end{itemize}

In this paper, we propose a state-driven approach to derive a class of dynamic graphon model that overcomes the three difficulties. In particular, we first show the proposed state-driven approach generalizes random graphs to graphons as the number of nodes goes to infinity under i.i.d. assumptions on the state random variables. Furthermore, we use the celebrated Aldous-Hoover representation to prove the proposed state-driven model covers $(\mathcal{W},\delta_{\square})$ via a bijection between graphons and permutation-invariant distributions on the universal graph space. Lastly, we show the resulting dynamic graphon model provides connections between the state dynamical system and the induced graphon dynamical system.\\

To the best of our knowledge, this work is the first to develop a generic dynamic model for graphons. The only work on developing a generic dynamic random graph model is $\cite{ZMN}$, which applies continuous Markov random processes to drive time-variant stochastic block models. Our state-driven model does not assume apriori the types of random process on the vertices and include arbitrary measurable edge-generating functions to form edges, and therefore is a generalization in both static setting and dynamic setting.\\

The rest of the report will be organized as the followings: Section 2 gives review for graphon and its equivalent constructions; Section 3 proposes our state-driven approach to derive dynamical systems on $(\mathcal{W},\delta_{\square})$, and shows the resulting dynamic graphon model after some generalization covers the dynamical system on $(\mathcal{W},\delta_{\square})$; Section 3 shows our state-driven model in a static setting is a generalization of the celebrated Erdős–Rényi model and stochastic block model; Section 4 gives some applications of the dynamic state-driven graphon model; Section 5 shows further research directions for this dynamic graphon model.\\

The terminologies in this paper will follow the standard ones in graph theory, graphon literature, and probability theory.\\

\section{Review Graphon}

In this section, we review the construction of $(\mathcal{W}, \delta_{\square})$ in two different aspects: (1) formal definition; (2) graph limit definition.\\

\subsection{Formal Definition}

Formally, $\mathcal{W} := \{w: [0,1]^2 \rightarrow [0,1]| \text{ be borel-measurable}\}$ as mentioned above. Given any $w \in \mathcal{W}$, $\delta_{\square}$ defines an equivalent class on $\mathcal{W}$ by $[w] := \{ w' \in \mathcal{W}: \delta_{\square}(w,w') = 0\}$. A graphon is an equivalent class on $\mathcal{W}$.\\

\begin{exmp}
Let $G = ([4],\{(1,2),(1,3)\})$ be a simple graph. The corresponding adjacency matrix and its corresponding graphon form are shown in Figure 1.

\begin{figure}
        \begin{tabular}{p{5cm}c}
            {$\displaystyle
                A_G = 
                {\renewcommand{\arraystretch}{1.5}
                \begin{pmatrix}
                    0 & 1 & 1 & 0\\
                    1 & 0 & 0 & 0\\
                    1 & 0 & 0 & 0\\
                    0 & 0 & 0 & 0\\
                \end{pmatrix}}
            $}
            &
            $\vcenter{\hbox{\includegraphics[scale=0.7]{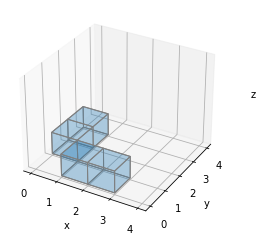}}}$
        \end{tabular}
        \caption{$A_G$ is the adjacency matrix of the simple graph $G := ([4],\{(1,2),(1,3)\})$, $z(x,y) := \mathbbm{1}_{([0,\frac{1}{4}] \times [\frac{1}{4},\frac{3}{4}]) \cup ([\frac{1}{4},\frac{3}{4}] \times [0,\frac{1}{4}])}$ is the corresponding graphon form of $A_G$.}
        \label{Figure 1}
\end{figure}
\end{exmp}

The formal definition does not generate any intuition to understand graphon. Moreover, the following result shows that a direct development of dynamics on $(\mathcal{W},\delta_{\square})$ is clumsy.\\

\begin{proposition}
If $p: [0,1] \times \mathcal{B}_{[0,1]} \rightarrow [0,1]$ is a Markov transition kernel, then the graphon chain generated by the following: $$w_{t+1}(x,y) := \int_{[0,1]^2} w_t(s,t)p(s,\{x\})p(t,\{y\})dsdt,$$ has a constant edge density.\\
\end{proposition}

\begin{proof}
By the definition of homomorphism density for graphon, let $E$ be the graph with two nodes and one edge between them, we have:

\begin{align*}
    t(E,w_{t+1}) & := \int_{[0,1]^2} w_{t+1}(x,y) dxdy\\
    & = \int_{[0,1]^2} \{\int_{[0,1]^2} w_t(s,t)p(s,dx)p(t,dy)dsdt\}\\
    & = \int_{[0,1]^2} w_t(s,t)\{\int_{[0,1]}p(s,dx)\}\{\int_{[0,1]}p(t,dy)\}dsdt\\
    & = \int_{[0,1]^2} w_t(s,t)dsdt\\
    & = t(E,w_t).
\end{align*}

Since it is true for any $t \in \{0,1,...,\}$, we have $t(E,w_t) = t(E,w_0), \forall t$.\\

\end{proof}

The above result shows that a direct development of dynamics on $(\mathcal{W},\delta_{\square})$ is non-trivial. The difficulty mainly comes from the definition of graphon as equivalent classes on $\mathcal{W}$, which are defined by $[w]:= \{w' \in \mathcal{W}: \delta_{\square}(w,w') = 0\}$, rather than simply measurable functions. Therefore, it is necessary to introduce the graph limit definition, which gives an intuitive explanation for graphon.\\

\subsection{Graph Limit Definition}

It is proved in $\cite{LL}$ that $t(\cdot, w_1) = t(\cdot, w_2) \iff \delta_{\square}(w_1,w_2) = 0$. Therefore, the graphon homomorphism density functions defines the same equivalent classes on $\mathcal{W}$ and therefore graphon.\\

To show graphon as the limit of graph sequence, we have to first define the homomorphism density functions for graphs and thereby the convergence of a graph sequence. In the rest of the section, let $G = (V,E)$ be a simple (i.e. undirected and unweighted) graph, and refer $V(G),E(G)$ to $V,E$ respectively.\\

\begin{definition}
Given two simple graphs $F,G$, a homomorphism function from $F$ to $G$ is a function $f : V(F) \rightarrow V(G)$ such that $(x,y) \in E(F) \implies (f(x),f(y)) \in E(G).$\\
\end{definition}

Therefore, we define $$\hom(F,G):= \{f : V(F) \rightarrow V(G)| (x,y) \in E(F) \implies (f(x),f(y)) \in E(G)\}$$ to be the set of all the homomorphism functions from $F$ to $G$. Now, we introduce the homomorphism density function, which is a key in defining graph limit.\\

\begin{definition}
Given two simple graphs $F,G$, a homomorphism density function from $F$ to $G$ is a function $t : F \times G \rightarrow [0,1]$ defined as below: 
\begin{align*}
    t(F,G) := \frac{|\hom(F,G)|}{|V(G)|^{|V(F)|}}.
\end{align*}
\end{definition}

It is not difficult to show that $t(\cdot,G) = t(\cdot,w_G), \forall G \in \mathcal{G}$. The example below gives a simple illustration of the equivalence between graph homomorphism density and the corresponding graphon homomorphism density.\\

\begin{exmp}
Let $G = ([4],\{(1,2),(1,3)\})$ as above, and $F := ([3],\{(1,2)(1,3)\})$. It is clear that $|V(G)|^{|V(F)|} = 4^3 = 64$. To calculate $|\hom(F,G)|$, we need to find all the maps that preserve the connectivity of $E(F)$. Therefore, one could verify that $\hom(F,G)$ consists of all maps that map $[3]$ to $[3]$ and $|\hom(F,G)| = 3! = 6$.\\

On the other hand, we could also use the graphon form of $G$, $w_G$, to calculate $t(F,G) = t(F,w_G)$. To simplify the equation below, we denote $E_F := ([0,\frac{1}{4}] \times [\frac{1}{4},\frac{3}{4}]) \cup ([\frac{1}{4},\frac{3}{4}] \times [0,\frac{1}{4}])$

\begin{align*}
    t(F,w_G) & := \int_{[0,1]^3} w_G(x_1,x_2)w_G(x_2,x_3) \prod_{i=1}^3 dx_i\\
    & = \int_{[0,1]^3} \mathbbm{1}_{E_F}(x_1,x_2) \mathbbm{1}_{E_F}(x_2,x_3) \prod_{i=1}^3 dx_i\\
    & = \int_{[0,1]^3} \mathbbm{1}_{([0,\frac{1}{4}] \times [\frac{1}{4},\frac{3}{4}] \times [0,\frac{1}{4}]) \cup ([\frac{1}{4},\frac{3}{4}] \times [0,\frac{1}{4}] \times [\frac{1}{4},\frac{3}{4}])}(x_1,x_2,x_3) \prod_{i=1}^3 dx_i\\
    & = \frac{2}{4}\frac{1}{4}\frac{1}{4}+ \frac{2}{4}\frac{1}{4}\frac{2}{4} = \frac{6}{4^3} = t(F,G)\\
\end{align*}

Figure 2 shows geometrically the meaning of the integration above: $t(F,w_G)$ is the volume of the shaded area.

\begin{figure}[h]
\centering
        \begin{tabular}{p{5cm}c}
            $\vcenter{\hbox{\includegraphics[scale=0.7]{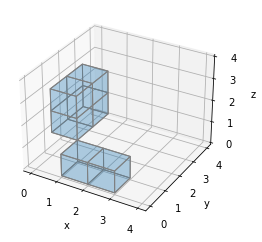}}}$
        \end{tabular}
        \caption{The volume of the shaded area is equal to $t(F,w_G) = t(F,G)$, the probability of drawing $|V(F)| = 3$ nodes from $V(G) = [4]$ and obtain a graph isomorphic to $F$}
        \label{Figure 2}
    \end{figure}

\end{exmp} 

Given graph homomorphism density functions, we are ready to define the limit of a graph sequence. Let $\{G_i\}_{i \in \mathbb{N}} = \{(V_i,E_i)\}_i$ be a sequence of simple graphs. We say that $\{G_i\} = \{(V_i,E_i)\}_i$ converges if the sequence $\{t(F,G_i)\}_i$ converges for all $F \in \mathcal{G}$, where $\mathcal{G}$ denotes the set of all simple graphs. Otherwise, we say that the graph sequence diverges.\\

\begin{definition}
Given a convergent sequence of graphs $\{G_i\}_i$, $w$ is the limit of the sequence if $$\lim_{n \rightarrow \infty} t(F,G_i) = t(F,w), \forall F \in \mathcal{G}.$$
\end{definition}

It turns out that $(\mathcal{W},\delta_{\square})$ consists of all the graph sequence limits defined by the homomorphism density functions as above. Conversely, given an arbitrary graphon, there exists a convergent graph sequence with limit equal to the graphon. For more detailed explanation of graph sequence convergence, the choice of graph parameters, and the construction of graphons, we refer readers to $\cite[p.173-193]{LL}$.\\

Remark: It is important to notice that the homomorphism density function is equal to the probability of uniformly drawing $|V(F)|$ number of vertices from $G$ and obtaining an sample sub-graph that is isomorphic to $F$. That is, $t(\cdot,G)$ is a distribution on $\mathcal{G}$ and is invariant under isomorphisms, which is permutation of graph nodes. Therefore, the study of graphon is an interesting intersection between the permutation-invariant distribution on random adjacency matrix in probability theory and graph theory. In fact, we prove in "State-Driven Dynamic Graphon" that there is a bijection map between the graphon space and the exchangeable random 2-dimensional array which has been studied in probability theory and statistics since 1920s.\\

\section{State-Driven Dynamic Graphon Model}

In this section, we first propose our state-driven approach to a dynamical systems on the finite random graph space, and extend naturally the dynamical system to the universal graph space. At each time stage, the model becomes a countable random graph model which we define below. Furthermore, we show the existence of a bijective map between $(\mathcal{W},\delta_{\square})$ and the countable random graph model. Thereafter, we use the bijective map and the Aldous-Hoover theorem to prove that for every time-variant graphon, there exists a corresponding state-driven dynamic graphon model after some generalization. That is, our state-driven approach provides broad enough models to cover the graphon space. Lastly, we show that the state-driven approach indeed rich structure to analyze the dynamical systems on graphon space.\\

Before proposing our state-driven approach, we need to first define the finite graph space and the universal graph space.\\

\begin{definition}
$(\mathcal{G}_{[N]},\mathcal{B}_{\mathcal{G}_{[N]}})$ is the finite graph space on $[N]$.\\
\end{definition}

Here, $[N]: = \{1,2,...,N\}, \forall N \in \mathbb{N}$, $\mathcal{G}_{[N]} := \{([N],e): e \in \{0,1\}^{[N]\times[N]}, e(x,x) = 0, e(x,y) = e(y,x), \forall x,y \in [N]\}$. Also, $\mathcal{B}_{\mathcal{G}_{[N]}} \subset 2^{\mathcal{G}_{[N]}}$ is an atomic sigma-algebra.\\

\begin{definition}
$(\mathcal{G}_{\mathbb{N}},\mathcal{B}_{\mathcal{G}_{\mathbb{N}}})$ is called the universal graph space.\\
\end{definition}

Where $\mathcal{G}_{\mathbb{N}}$ and $\mathcal{B}_{\mathcal{G}_{\mathbb{N}}}$ are defined analogously as in the finite graph space above. Now, we give a formal definition of the countable random graph model.\\

\begin{definition}
A countable random graph model is a probability measure $P: \mathcal{B}_{\mathcal{G}_{\mathbb{N}}} \rightarrow [0,1]$ that is invariant under permutations of $\mathbb{N}$.\\
\end{definition}

For notation in this work, let $(\Omega,\mathcal{F}, \mathbb{P})$ be an underlying probability space, $(\mathcal{X},\mathcal{B}_{\mathcal{X}})$ be a state measurable space. $\{(V_{i,t})_t\}_{i \in \mathbb{N}}$ be the set of random processes, where for all $i \in \mathbb{N}, t \in T$, $V_{i,t}: \Omega \rightarrow \mathcal{X}$ is a random variable.\\

Our state-driven approach to generate dynamics on graphon space is the following:\\

\begin{itemize}
    \item Assign each node $i \in \mathbb{N}$ on the graph with a corresponding state random process $(V_{i,t})_{t \in T}$;
    \item Pick an measurable function $e: \mathcal{X} \times \mathcal{X} \rightarrow \{0,1\}$, and for each unordered pair $(V_i,V_j)$, each $t \in T$, assign an edge random variables $Y_{ij,t} \sim Bern(e(V_{i,t},V_{j,t}))$;
    \item Construct the dynamics of countable random graph model via state-driven approach;
    \item Prove there exists a bijection between the countable random graph model and the graphon space;
    \item Derive a dynamical system on the graphon space from the isomorphism and the dynamics of countable random graph model.\\
\end{itemize}

We call our approach state-driven, because the dynamics on the graphon space is induced by the state random processes on the nodes and corresponding edge-generating random variables.\\

Before developing the dynamical systems, we need to first show the state-driven model induces a consistent sequence of probability measures on the sequence of finite graph spaces indexed by $N$ and thereby induces a probability measure on the universal graph space at any fixed time $t$.\\

Given $(\Omega,\mathcal{F},\mathbb{P})$ and $(\mathcal{X},\mathcal{B}_{\mathcal{X}})$ as above, we generate a finite set of random variables $\{V_i\}_{i = 1}^N$, where $V_i:\Omega \rightarrow \mathcal{X}$, for each $i \in \{1,...,N\}$. Therefore, we have a finite set of induced probability measures on $(\mathcal{X},\mathcal{B}_{\mathcal{X}})$: $\{\mu_i\}_{i = 1}^N$ that are defined by $\mu_i(E) := \mathbb{P}(\{V_i^{-1}(E)\})$, for all $E \in \mathcal{B}_{\mathcal{X}}, i \in \mathbb{N}$.\\

Now, in order to generate a random graph from random variables $\{V_i\}_{i=1}^N$, it is natural for us to define edge-generating functions $$\mathcal{E}:= \{e \in [0,1]^{\mathcal{X} \times \mathcal{X}}: \forall x,y \in \mathcal{X}, e(x,y) = e(y,x), e(x,\cdot), e(\cdot,y) \text{ measurable}\}.$$ Therefore, for each realization of $V := \{V_i\}_{i = 1}^N$, each edge-generating function $e$ induces a random graph model denoted by $G_e(V) : = \{([N],E): (i,j) \in E \iff Y_{ij} = 1\}.$ Here, $Y_{ij} \sim Bern(e(V_i,V_j))$. In the rest of the paper, we will abuse the notation $G_e(V)$ to represent the conditional probability measure on $(G_{[N]}, \mathcal{B}_{G_{[N]}})$: $\forall A \in \mathcal{B}_{G_{[N]}}$, $$P_e(A | V) := \prod_{(i,j)}e(V_i,V_j)^{A_{ij}}(1-e(V_i,V_j))^{1-A_{ij}}.$$


There are three generalizations to make to the current model with a fixed $e \in \mathcal{E}$ for all pairs of $(V_i,V_j)$:\\

\begin{itemize}
    \item[i] pick different edge-generating function $e_{ij} \in \mathcal{E}$ for different pairs of $(V_i,V_j)$;
    \item[ii] assign a distribution on $\mathcal{E}$ and draw $e_{ij}$ from that distribution for each pair of $(V_i,V_j)$;
    \item[iii] define graph-generating functions $g: X^N \rightarrow \mathcal{G}_N{[N]}$ to have a more general setting and hence more interesting results.\\
\end{itemize}

We will show that generalization (2) is broad enough to generate $(\mathcal{W},\delta_{\square})$. That is, if we add one more level of randomness on the choice of edge-generating functions, our state-driven model cover all graphons.\\

Now, given the finite graph space $(\mathcal{G}_{[N]},\mathcal{B}_{\mathcal{G}_{[N]}})$ on $[N]$, fix an edge-generating function $e$ and the finite set of random variables $V$ as before, and let the event $E_{ij} = \{e(V_i,V_j) = 1\}$, we obtain the following result:\\

\begin{lemma} $G_e(V)$ induces a probability measure, denoted by $P_{G_e(V)}$, on $(\mathcal{G}_{[N]},\mathcal{B}_{\mathcal{G}_{[N]}})$ for each $N < \infty$.\\
\end{lemma}

\begin{proof}
For any fixed $N < \infty$, let $A \in \mathcal{G}_{[N]}$, we have:
\begin{equation*}
    \begin{aligned}
    P_{G_e(V)}(A) &= \int_{\mathcal{X}^N} P_e(A|v)\mu(dv).
    \end{aligned}
\end{equation*}

\noindent Where $\mu$ is the measure of the random vector $V = (V_1,...,V_N)$. Since $\mu$ is a probability measure on $\mathcal{X}^N$ and $\mathbb{P}(\cdot|v)$ is a probability measure on $\mathcal{B}_{\mathcal{G}_{[N]}}$, we are done.\\
\end{proof}

Therefore, we have shown that $(\mathcal{G}_{[N]},\mathcal{B}_{\mathcal{G}_{[N]}}, P_{G_e(V)})$ is a well-defined finite graph probability space. We call it State-driven Random Graph Space because the randomness and the probability measure on the finite graph space $(\mathcal{G}_{[N]},\mathcal{B}_{\mathcal{G}_{[N]}})$ come from the states of the vertices $V$.\\

Now, we start to develop a dynamical system based on $(\mathcal{G}_{[N]}, \mathcal{B}_{\mathcal{G}_{[N]}})$. By Lemma 1, we have that for each fixed $t \in T$, $V_t := \{V_{i,t}\}_{i \in [N]}$ induces a probability measure $P_{G_e(V_t)}$ on the $(\mathcal{G}_{[N]},\mathcal{B}_{\mathcal{G}_{[N]}})$. Let $\mathcal{G}_{[N]}^\mathbb{N}$ denotes the set of all possible discrete trajectories on the product $[N]$-graph space, $\mathcal{B}_{\mathcal{G}_{[N]}}^{\otimes \mathbb{N}}$ denotes the completion of the product of sigma-algebras $\mathcal{B}_{\mathcal{G}_{[N]}}^{\mathbb{N}}$, and $\otimes P_{G_e(V_t)}$ denote the completion of $\prod_{t =1}^{\infty} P_{G_e(V_t)}$, we then have from Kolmogorov extension theorem that $(\mathcal{G}_{[N]}^\mathbb{N},\mathcal{B}_{\mathcal{G}_{[N]}}^{\otimes \mathbb{N}},\otimes P_{G_e(V_t)})$ is a well-defined probability measure space. One can easily check that the consistency assumptions are satisfied due to the definition of $P_{G_e(V_t)}$. Since every discrete random process is a random variable on the infinite product space, we have proved the following result:\\

\begin{theorem}
Given any finite set of random processes $\{(V_{i,t})_{t \in \mathbb{N}}\}_{i \in [N]}$, and an edge-generating function $e: X \times X \rightarrow [0,1]$, $\otimes P_{G_e(V_t)}$ defines a discrete random process on the space $(\mathcal{G}_{[N]},\mathcal{B}_{\mathcal{G}_{[N]}})$.\\
\end{theorem}

Note: The authors only consider discrete-time processes in this paper. But it is clear that one could formulate continuous random graph processes and dynamical countable random graph model by choosing continuous random processes on the state space. The authors will generalize the analysis to continuous processes in future papers.\\

Therefore, by assigning each vertex with a state random process, and fixing an edge-generating function $e$, we have shown that the state-driven approach successful derive a random process on $(\mathcal{G}_{[N]},\mathcal{B}_{\mathcal{G}_{[N]}})$. It is not hard to see that we have built the model in a way such that it can be extended to the graphs with countably infinite vertices, which have a close relationship with graphon or graph limit.\\

The following result gives an $e$-robust condition on the state random variables such that the proposed state-driven approach generates a countable random graph model. That is, for any given fixed $e \in \mathcal{E}$, $P_{G_e(V)}$ is permutation-invariant probability measure on $(\mathcal{G}_{\mathbbm{N}},\mathcal{B}_{\mathcal{G}_{\mathbbm{N}}})$. The reason for $e$-robustness condition is to study the edge-generating law of a randomly networked system and predict future network, given state random processes.\\

\begin{theorem}
$\forall e \in \mathcal{E}, \forall \sigma \in Sym(\mathbbm{N})$, $P_{G_e(V)}$ is $\sigma$-invariant if and only if $V$ consists of independent random variables that are equal $\mathbbm{P}$-a.e..\\
\end{theorem}

\begin{proof}
First assume $V$ consists of independent random variables that are equal $\mu$-a.e.. It follows that $\forall \sigma$, $\forall A \in \mathcal{B}_{\mathcal{G}_{\mathbbm{N}}}$,

\begin{align*}
    P_{G_e(V)}(\sigma(A)) & = \int_{\mathcal{X}^{\mathbbm{N}}}P_e(\sigma(A)|v) \prod_{i \in \mathbbm{N}} \mu_1(dv_i)\\
    & = \int_{\mathcal{X}^{\mathbbm{N}}}P_e(\sigma(A)|\sigma(v)) \prod_{i \in \mathbbm{N}} \mu_1(dv_{\sigma(i)})\\
    & = \int_{\mathcal{X}^{\mathbbm{N}}}P_e(A|v) \prod_{i \in \mathbbm{N}} \mu_1(dv_{i})\\
    & = P_{G_e(V)}(A).
\end{align*}

The second equality comes from

\begin{align*}
    &\prod_{(i,j)}e(v_i,v_j)^{A_{{\sigma(i)}{\sigma(j)}}}(1-e(v_i,v_j))^{1-A_{{\sigma(i)}{\sigma(j)}}} \prod_{i \in \mathbbm{N}} \mu_1(dv_i)\\
    = & \prod_{(i,j)}e(v_{\sigma(i)},v_{\sigma(j)})^{A_{{\sigma(i)}{\sigma(j)}}}(1-e(v_i,v_j))^{1-A_{{\sigma(i)}{\sigma(j)}}} \prod_{i \in \mathbbm{N}} \mu_1(dv_{\sigma(i)}).
\end{align*}

Conversely, assume $\exists B$ such that $\mu(B) > 0$, and $\exists i \neq j \neq k \in \mathbbm{N}$ such that $\mu_{ki}(P_{ki}(B)) \neq \mu_{kj}(P_{kj}(B))$, then let $e \in \mathcal{E}$ such that 

\begin{equation*}
    \begin{aligned}
        \begin{cases}
        e(v_k,v_i) = e(v_k,v_j) = 1 & \text{ on } B\\
        e(v_k,v_i) = e(v_k,v_j) = 0 & \text{ on } B^c
        \end{cases}
    \end{aligned}
\end{equation*}

Finally, let $\sigma(i) = \sigma(j)$, $\sigma(k) = \sigma(k)$, and $A = A_{ki} = 1$, it follows

\begin{align*}
    & \int_{\mathcal{X}^{\mathbbm{N}}} P_e(A_{ki}|v) \mu(\prod dv_i) - \int_{\mathcal{X}^{\mathbbm{N}}} P_e(A_{\sigma(k)\sigma(i)}|v) \mu(\prod dv_i)\\
    = & \int_{\mathcal{X}^2} e(v_k,v_i) \mu_{ki}(dv_k dv_i) - \int_{\mathcal{X}^2} e(v_k,v_j) \mu_{kj}(dv_k dv_j)\\
    = & \int_{\mathcal{X}^2} e(v_k,v_i) \mu_{ki}(dv_k dv_i) - \int_{\mathcal{X}^2} e(v_k,v_j) \mu_{kj}(dv_k dv_j)\\
    = & \mu_{ki}(B) - \mu_{kj}(B) \neq 0.
\end{align*}

\end{proof}

The theorem above gives us both sufficient and necessary conditions for a countably infinite state random variables to induce a countable random graph model, given any $e \in \mathcal{E}$.\\

Furthermore, it is not hard to see the possibility of extending $\mathcal{E}$ to asymmetric functions, given the events on the product state space resulting the asymmetry lie within the null set of the joint measure $\mu$.\\

\begin{proposition}
Given a sequence of random variables $V = \{V_i\}_{i \in \mathbbm{N}}$, any function in the following set $$\{e \in [0,1]^{\mathcal{X} \times \mathcal{X}}|e(\cdot,\cdot) \text{ bi-measurable, } \mu(\{(x,y)|e(x,y)\neq e(y,x)\}) = 0\}$$ results in a countable random graph model $P_{G_e(V)}$.\\
\end{proposition}

Before the third main result, which proves the key connection between graphon and the countable random graph model and thereby constructs the proposed state-driven dynamic graphon model, we state a lemma from $\cite{TT}$. We will use the lemma to prove our second main result.\\

\begin{lemma}[$\cite{TT}$]
For any countable random graph model $P$, there exists a sequence of simple graphs $\{G_i\}_{i = 1}^{\infty}$ such that $\lim_{i \rightarrow \infty} t(\cdot,G_i) = P(\cdot)$ on $\mathcal{G}$.\\
\end{lemma}

\begin{proof}
This is the Graph Correspondence Principle, Proposition 3.4, in $\cite{TT}$.\\
\end{proof}

\begin{theorem}
There exists a bijective map between graphons and countable random graph models.\\
\end{theorem}

\begin{proof}
Let $w \in (\mathcal{W},\delta_{\square})$ be arbitrary. By Corollary 10.34 in \cite[p.170]{LL}, $$\delta_{\square}(w_1,w_2) = 0 \iff t(G,w_1) = t(G,w_2), \forall G \in \mathcal{G},$$ the map $W \rightarrow t(\cdot,w)$ is a bijection. Also, it is not hard to verify that $t(\cdot,w)$ is a permutation-invariant local probability measure on $\mathcal{G}$ and therefore a countable random graph model. We conclude the map $w \rightarrow t(\cdot,w) \rightarrow P$ is an injective map from graphons to countable random graph models. Again, we don't differentiate $G \in \mathcal{G}$ for $t(\cdot,w)$ and $A \in \mathcal{B}_{G_{\mathbbm{N}}}$ for $P$.\\

On the other hand, let $P: \mathcal{B}_{G_{\mathbbm{N}}} \rightarrow [0,1]$ be an arbitrary countable random graph model. By Lemma 3, $\exists \{G_i\}_{i \in \mathbb{N}}$ such that $t(\cdot,G_i) = P(\cdot)$. Therefore, $\exists w \in (\mathcal{W},\delta_{\square})$ such that $t(\cdot,W) = \lim_{n \rightarrow \infty} t(\cdot,G_n) = t(\cdot,w)$. Hence, the map $w \rightarrow t(\cdot,w) \rightarrow P$ is surjective. We are done.\\
\end{proof}

Intuitively, graphons and countable random graph models are two explicit forms of the probability distributions on the set of simple graphs $\mathcal{G}$. With the above theorem, we prove a one-to-one relationship between the two explicit representative forms that are different in nature.\\

If $\{G_e(V_t)\}_{t \in \mathbb{N}}$ is a dynamical countable random graph model, we can consider it directly as a dynamical graphon $\{w_{e,t}\}_{t \in \mathbb{N}}$, where $w_{e,t}: [0,1] \times [0,1] \rightarrow [0,1]$ is a graphon for a fixed $t$, due to the existence of the above bijection $w_{e,t} \rightarrow t(\cdot,w_{e,t}) \rightarrow P_{G_e(V_t)}(\cdot) \rightarrow G_e(V_t)$.\\

Therefore, it remains to show $\{G_e(V_t)\}_{t \in \mathbb{N}}$ is a well-defined countable random graph model process for all $e \in \mathcal{E}$ and all $V_t = V_t^\mathbb{N}$. But for each fixed $t \in \mathbb{N}$, such $G_e(V_t)$ indeed define a countable random graph model by Theorem 2. Furthermore, since $(V_{i,t})_t$ are all random processes and the consistency conditions in the Kolmogorov extension theorem are satisfied, $G_e(V_t)$ also satisfies the consistency conditions. Therefore, we have the following result:\\

\begin{corollary}
For fixed $t \in \mathbb{N}$, $G_e(V_t)$ generates a graphon $w_{e,t}$ for any $e \in \mathcal{E}$ if and only if $\{V_{i,t}\}_{i \in \mathbb{N}}$ consists of independent random variables that are equal $\mu_t-a.e.$.\\
\end{corollary}

\begin{proof}
It follows from the injective map $P_{G_e(V_t)} \rightarrow t(\cdot,w_{e,t}) \rightarrow w_{e,t}$ and Theorem 2.\\
\end{proof}

Note: the condition is easy to satisfy. For example, any i.i.d sequence of random variables would satisfy it.\\

Now, we show that the generalization (2) is enough for the proposed state-driven dynamic graphon model to cover all graphons. To that end, we need the celebrated Aldous-Hoover representation for jointly exchangeable 2-dimensional random array in 1983.\\

\begin{definition}
Given $\{X_{(i,j)}\}_{(i,j) \in \mathbbm{N}^2}$ be a 2-dimensional random arrays. If $\{X_{(i,j)}\}_{(i,j)} =^{d} \{X_{(\sigma(i),\sigma(j))}\}_{(i,j)}, \forall \sigma \in Sym(\mathbb{N})$, then we say that $\{X_{(i,j)}\}_{(i,j)}$ is a jointly exchangeable 2-dimensional random arrays.\\
\end{definition}

Here, we only states the jointly-exchangeable 2-dimensional random array as it is directly related to graphon and countable random graph model. For more generalized results on the k-dimensional exchangeable random arrays, we refer readers to $\cite{TA}$ for detailed explanation.\\

Let $\Tilde{\mathcal{E}} := \{e \in [0,1]^{[0,1]^2}: e(x,y) = e(y,x), \forall x,y \in [0,1]\}$, $\{w(U_i,U_j)\}_{(i,j)}$ be the countable random graph model where $U_i$ are i.i.d. random variables from $Unif([0,1])$ and $P(\{(i,j) \in E\})$ is $Bern(w(U_i,U_j))$.\\

\begin{lemma}[Aldous-Hoover]
Let $X := \{X_{ij}\}_{(i,j) \in \mathbb{N}^2}$ be a jointly exchangeable 2-dimensional binary random array, then exists a probability measure $\mu_w$ on $\Tilde{\mathcal{E}}$, such that for any $E \in \mathcal{B}_{\mathcal{G}_{\mathbb{N}}}$ $$\mathbb{P}(X \in E) = \int_{\Tilde{\mathcal{E}}} \mathbb{P}(\{w(U_i,U_j)\}_{(i,j) \in \mathbb{N}^2} \in E) \mu_w(dw).$$
\end{lemma}

Therefore, it follows from the Aldous-Hoover representation theorem, the proposed state-driven dynamic graphon model covers all the possible graphon trajectories, by adding randomness on the edge-generating function.\\

Now, we propose a generalization to the original state-driven dynamic graphon model. In particular, we add another level of randomness on the the choice of $e$ by placing a distribution on $\mathcal{E}$. For each fixed $e \in \mathcal{E}$, we apply the original state-driven approach to obtain $P_{G_e(V_t)}$. Now, we are ready to show that the generalized state-driven dynamic graphon model cover all graphons.\\

\begin{theorem}
For every dynamic graphon $(w_t)_t$, there exists an i.i.d. sequence of random processes $\{(V_{i,t})_t\}_i$ and a probability measure $\mu_e$ on $\mathcal{E}$, such that the generalized state-driven dynamic graphon model induces a of probability measures $\{P_{G_e(V_t)}\}_{w \sim \mu_e}$ on $\mathcal{B}_{G_{\mathbbm{N}}}$ that coincides with $t(\cdot,w)$ on $\mathcal{G}$.\\
\end{theorem}

Thus, we showed that the proposed state-driven approach is natural from the graph limit definition perspective. The above theorem also shows that a generalization of the state-driven graphon model is sufficient to generate all $w \in (\mathcal{W},\delta_{\square})$.\\

In the rest of the current section, we show that the proposed state-driven graphon model is a generalization of the Erdős–Rényi model and the stochastic block model.\\

\begin{exmp}[Erdős–Rényi model]
Let $V := \{V_i\}_{i =1}^n$ be a finite set of i.i.d. random variables, $e \equiv p$, then $G_e(V)$ on $\mathcal{B}_{G_{\mathbbm{n}}}$ is equal to the $G(n,p)$ model on $\mathcal{G}_{n}$.\\
\end{exmp}

\begin{exmp}[stochastic block model]
Let $V := \{V_i\}_{i =1}^n$ be a finite set of i.i.d. random variables with discrete state space $[r]$, $e := \{p_{ij}\}_{1 \leq i,j \leq r}$, then $G_e(V)$ on $\mathcal{B}_{G_{\mathbbm{n}}}$ is equal to the stochastic block model on $\mathcal{G}_{n}$ with $C_k = \{i: V_i = k\}$ and edge probability $p$.\\
\end{exmp}

\section{Dynamic Graphon Analysis}

One reason of developing the state-driven approach is to study the behavior of the dynamic random network via the knowledge of the state random processes. In this section, we explore some useful properties of the state-driven dynamic graphon:\\

\begin{itemize}
    \item If $(V_t)_t$ is a stationary process, then the positive possibility of a graph $A \in \mathcal{B}_{G_{\mathbbm{N}}}$ occurring at some stage implies that $A$ occurs infinitely often.
    \item If $(V_t)_t$ is a weakly-mixing process, then the time average of the graphon model converges to a specific random graph model that is invariant of the state-randomness.\\
\end{itemize}

\begin{definition}
A discrete-time dynamical system is generated by iterations of a map $f : X \circlearrowleft$. For any $x \in X$, $\mathcal{O}_+(x) := \{x,fx,f^2x,...\}$ is called a forward orbit of $x$.\\
\end{definition}

Given a random process $(V_t)_t$, let $(\mathcal{X}^{\mathbbm{N}},\mathcal{B}_{\mathcal{X}}^{{\otimes \mathbbm{N}}},\mathbbm{P} \circ V_t^{-1})$ be the space of actions ($X$) in the above definition. Also, let $T: (\mathcal{X}^{\mathbbm{N}},\mathcal{B}_{\mathcal{X}}^{{\otimes \mathbbm{N}}},\mathbbm{P} \circ V_t^{-1}) \circlearrowleft$ be the shift operator: $$T(\{A_1,A_2,A_3,...\}) := \{A_2,A_3,A_4,...\}.$$ That is, $T$ push the time forward by one discrete period.\\

Now, let $C := \{A_1,A_2,...,A_n\} \in \mathcal{B}_{\mathcal{X}}^{{\otimes \mathbbm{N}}}$ be an arbitrary cylinder set. It follows that $$\mathbbm{P} \circ V_t^{-1}(T^{-k}(C)) = \mathbbm{P}(V_{k+1} \in A_1,V_{k+2} \in A_2,...,V_{k+n} \in A_n).$$

To prove the first statement, we need to define measure-preserving transformation (mpt).\\

\begin{definition}
Given any two probability spaces $(X,\mathcal{B}_X,\mu_x)$ and $(Y,\mathcal{B}_Y,\mu_y)$, a map $f: X \rightarrow Y$ is an mpt if $\forall B \in \mathcal{B}_Y, \mu_y(B) = \mu_x(f^{-1}(B))$.\\
\end{definition}

Now, we show that if the state random process is stationary, then for a reasonable choice of $e \in \mathcal{E}$, the resulting dynamic graphon model $w_{e,t}$ has the tendency to repeat its history.\\

\begin{theorem}
Given a dynamic graphon model $(w_{e,t})_t$ that is generated by a sequence of i.i.d. stationary random processes $(V_t)_t$ and $e \in \mathcal{E}$, it follows: $$\exists t < \infty, t(G,w_{e,t}) > 0 \implies (w_{e,t})_t|_{|V(G)|} \text{ goes back to $G$ infinitely often. }$$
\end{theorem}

Here, an arbitrary sampling dynamic $|V(G)|$-subgraph from $(w_{e,t})_t$ is denoted by $(w_{e,t})_t|_{|V(G)|}$.\\

\begin{proof}
By the assumption, we have $(V_t)_t$ consists of i.i.d. stationary random processes. That is, $\forall i \in \mathbbm{N}$, the shift operator $T: (\mathcal{X},\mathcal{B}_{\mathcal{X}}, \mathbbm{P} \circ V_{i,t}^{-1}) \circlearrowleft$ is an mpt.\\

Now, given any $G \in \mathcal{G}$, we find the corresponding $A_G \in \mathcal{B}_{G_{\mathbbm{N}}}$. Also, $t(G,w_{e,t}) > 0 \implies P_{G_e(V_t)}(A_G) > 0$. But

\begin{align*}
    P_{G_e(V_t)}(A_G) & = \int_{\mathcal{X}^{\mathbbm{N}}} P_e(A_G|v_t) \otimes_{i \in \mathbbm{N}} \mu_{1,t}(\prod_{i \in \mathbbm{N}} dv_{i,t})\\
    & = \int_{\mathcal{X}^{|V(G)|}} P_e(A_G|v_{1,t},...,v_{|V(G)|,t}) \otimes_{i \in [|V(G)|]} \mu_{1,t}(\prod_{i \in [|V(G)|]} dv_{i,t}),
\end{align*}

where the fist equation comes from $\forall i \in \mathbbm{N}, \mu_{i,t} = \mu_{1,t}$, the second follows from $|V(G)| < \infty$ and permutation-invariance. Therefore, $P_{G_e(V_t)}(A_G) > 0$ implies $\exists B_G \in \mathcal{B}_{G_{[|V_{G}|]}}$ such that $\prod_{i \in [|V(G)|]} \mu_1 (B_G) > 0$ and $$P_e(A_G|v_1,...,v_{|V(G)|}) > 0, \forall (v_1,...,v_{|V(G)|}) \in B_G.$$ Since the induced transformation of finite product of mpts is still an mpt on the finite product space, it follows from the Poincaré recurrence theorem that $(v_{1,t},...,v_{|V(G)|,t})$ goes back to $B_G$ infinitely often.\\

Lastly, $P_e(A_G|v_{1,t},...,v_{|V(G)|,t}) > 0$ implies strictly positive probability of $A_G$ for every time $(v_{1,t},...,v_{|V(G)|,t}) \in A_G$. Therefore, we conclude that any $(w_{e,t})_t|_{|V(G)|}$, which equals to $G_e(V_{1,t}...,V_{|V(G)|,t})$ by permutation-invariance, goes back to $G$ infinitely often.

\end{proof}

The above theorem tells us that the history of the dynamic graphon model will repeat itself infinitely many times, if the model is driven by i.i.d. stationary processes.\\

Moreover, if we further assume the state random processes to be weakly mixing, we have a quantitative measure of the frequency for a specific graph to return to the sample graph trajectory.\\

\begin{definition}
An mpt $f: X \rightarrow X$ is called weakly-mixing if $\forall A, B \in \mathcal{B}_X$, there exists $J(A,B) \subset \mathbbm{N}$ with density 0 such that $$\frac{1}{n}\mu_X\{f^{-n} A \cap B\} \rightarrow \mu_X(A)\mu_X(B),\text{ as } n \rightarrow \infty, n \notin J,$$ where $J \subset \mathbbm{N}$ has density 0 if $$\frac{|J \cap [N]|}{n} \rightarrow 0.$$
\end{definition}

\begin{theorem}
Given a dynamic graphon model $(w_t)_t$ that is generated by a squence of i.i.d. weakly mixing random processes $(V_t)_t$. For any $e \in \mathcal{E}$, $G \in \mathcal{G}$, as $n \rightarrow \infty$, we have $$\frac{1}{n} \sum_{i = 0}^{n-1} P_e(A_G|v_0) \circ T^i \longrightarrow t(G,w_0) \text{ a.e.} .$$
\end{theorem}

\begin{proof}
For any fixed $G \in \mathcal{G}$, we can find a corresponding $A_G \in \mathcal{B}_{|V(G)|}$ due to the permutation-invariance of $P_{G_e}(V_t)$. Now, $\mu_{1,t}$'s are weakly mixing with respect to the shift operator $T$ implies that $\otimes_{i \in [|V(G)|]} \mu_{1,t}$ is ergodic with respect to $T$. Moreover, $P_e(A_G|v_0) \in L^1(\otimes_{i \in [|V(G)|]} \mu_{1,0})$. It follows from Birkhoff Ergodic Theorem that $\frac{1}{n} \sum_{i = 0}^{n-1} P_e(A_G|v_0) \circ T^i$ converges almost surely to $$\int_{\mathcal{X}^{|V(G)|}} P_e(A_G|v_{1,0},...,v_{|V(G)|,0}) \otimes_{i \in [|V(G)|]} \mu_{1,0}(\prod_{i \in [|V(G)|]} dv_{i,0}),$$ which equals to $t(G,w_0)$ by definition of $A_G$. We are done.\\
\end{proof}

The above lemma tells us a classic result from ergodic theory that the time average is equal to the space average. That is, if we obtain the space average by designing state random processes and edge-generating functions, we can have the frequency of a graph showing on an arbitrary sample graph trajectory.\\

\begin{corollary}
If the state random processes $(V_t)_t$ are i.i.d. Markov with a countably infinite state space, $|\mathcal{X}| = \aleph_0$, then for any $e \in \mathcal{E}$, $G \in \mathcal{G}$, as $n \rightarrow \infty$, we have $$\frac{1}{n} \sum_{i = 0}^{n-1} P_e(A_G|v_0) \circ T^i \longrightarrow t(G,w_0) \text{ a.e.} .$$ $t(G,w_0)$ is independent of the initial distribution.\\
\end{corollary}

\begin{proof}
It follows directly from lemma 4. and the fact that Markov random processes with countably infinite state space is mixing.\\
\end{proof}

Therefore, we have shown that powerful analytical results are available for the behavior of the resulting dynamic graphon by making assumptions on the state random processes in our state-driven dynamic graphon model.\\

\section{Conclusion}

This paper develops a state-driven approach to both random graph processes and dynamical graphon. The state-driven approach first assigns random processes to $\mathbbm{N}$ and thereby generate time-indexed sequence of probability measures on the finite graph space, then applies the permutation-invariance property of the countable random graph model to extend the generated time-indexed sequence of the permutation-invariant probability measures to the universal graph space, and finally uses the bijection between the permutation-invariant probability measures to the universal graph space and $(\mathcal{W},\delta_{\square})$ to obtain dynamic graphon.\\

More importantly, we show from the following three aspects to show that the state-driven dynamic graphon model is a practical approach to generate dynamic graphon.\\

\begin{itemize}
    \item (1) graph limit definition preservation: the model constructs a time-index sequence of permutation-invariant probability measures on the universal graph space, which is shown to be equivalent to a graphon at each fixed time, and thereby preserves the graph limit definition of graphon.
    \item (2) genericity: the model can be easily generalized by adding a distribution to the set of all edge-generating functions before using the determined edge-generating function to generate the i.i.d. Bernoulli edge random variables. Thereafter, an application of Aldous-Hoover representation lemma shows the state-driven model covers $(\mathcal{W},\delta_{\square})$.
    \item (3) analysis availability: due to the separation of randomness to the state random processes and the edge random variables, the state-driven graphon model is closer to real-world randomly connected systems and more convenient to analyze the interaction between the state random processes and edge random variables.\\
\end{itemize}

Future directions on the state-driven dynamic graphon includes large network formation and modeling large randomly connected dynamical systems. We are interested in studying the influence of the resulting graphon to the state random processes. One example is the SIR model in the study of epidemics \cite{YRYP}. For example, one can apply the dynamical graphon to modify the SIR model so that the model could reflect the interaction between the topology of a time-variant large network and the contact process on it. Moreover, since the state random variable could represent economic status, geometric distance, or social distancing, it is not hard to put control in the state process to influence the network and therefore drive the dynamics to the optimal trajectory for some objective function.\\

Finally, the work focuses on a discrete-time dynamical system setting. Continuous-time randomly networked dynamical systems are remained for further study.\\

\end{document}